\documentclass{amsart}
\usepackage[T1]{fontenc}

\usepackage{amsmath}
\usepackage{amssymb}
\usepackage{bm}
\usepackage{mathrsfs}
\usepackage{accents}

\usepackage{float}
\usepackage{graphicx,color}
\usepackage{ascmac}

\usepackage{tabularx}

\usepackage{tikz}
\usetikzlibrary{knots,calc,decorations.markings}

\usepackage{amsthm}
\theoremstyle{definition}
\newtheorem{THM}{Theorem}
\newtheorem{LEM}[THM]{Lemma}
\newtheorem{PROP}[THM]{Proposition}
\newtheorem{COR}[THM]{Corollary}
\newtheorem{DEF}[THM]{Definition}
\newtheorem{RMK}[THM]{Remark}

\newtheorem{QUE}[THM]{Question}

\newtheorem*{THM*}{Theorem}
\newtheorem*{LEM*}{Lemma}
\newtheorem*{PROP*}{Proposition}
\newtheorem*{COR*}{Corollary}
\newtheorem*{DEF*}{Definition}
\newtheorem*{RMK*}{Remark}
\newtheorem*{EX*}{Example}

\numberwithin{figure}{section}
\numberwithin{equation}{section}
\numberwithin{THM}{section}

\newcommand{\SK}{\mathcal{S}_q(1^{\otimes 2n})}

\title[On power subgroups of Dehn twists in $\Delta(g,0)$]{On power subgroups of Dehn twists in hyperelliptic mapping class groups}
\author{Wataru Yuasa}
\address{Department of Mathematics\\
  Tokyo Institute of Technology\\
  2-12-1 Ookayama, Meguro-ku, Tokyo 152-8551, Japan}
\email[]{yuasa.w.aa@m.titech.ac.jp}


\begin{document}
\begin{abstract}
This paper contains two topics, the index of a power subgroup in the mapping class group $\mathcal{M}(0,2n)$ of a $2n$-punctured sphere and in the hyperelliptic mapping class group $\Delta(g,0)$ of an oriented closed surface of genus $g$. 
The main tool is a projective representation of $\mathcal{M}(0,2n)$ obtained through the Kauffman bracket skein module. 
For $\mathcal{M}(0,2n)$, we prove that the normal closure of the fifth power of a half-twist has infinite index. This is the remaining case of a Masbaum's work. 
For $\Delta(g,0)$, we consider the normal closure of $m$-th powers of Dehn twists along all symmetric simple closed curves. 
We show the subgroup has infinite index if $m\geq 5$ and $m\neq 6$ for any $g\geq 2$.
\end{abstract}
\maketitle

\tikzset{->-/.style={decoration={
  markings,
  mark=at position #1 with {\arrow[black,thin]{>}}},postaction={decorate}}}
\tikzset{-<-/.style={decoration={
  markings,
  mark=at position #1 with {\arrow[black,thin]{<}}},postaction={decorate}}}
\tikzset{-|-/.style={decoration={
  markings,
  mark=at position #1 with {\arrow[black,thin]{|}}},postaction={decorate}}}
\tikzset{
    triple/.style args={[#1] in [#2] in [#3]}{
        #1,preaction={preaction={draw,#3},draw,#2}
    }
}

\section{Introduction}
In this paper, the main subject is ``power subgroups'' in a mapping class group. 
We will study indices of these subgroups by using a projective representation and linear skein theory.
The mapping class group $\mathcal{M}(g,p)$ is the group of all isotopy classes of orientation-preserving homeomorphisms of an oriented closed surface of genus $g$ with $p$ punctures. 
$\mathcal{M}(0,2n)$ is generated by half-twists $h_i$ ($i=1,2,\dots,2n$) permuting the $i$-th and the $(i+1)$-th punctures. 
The first part of the present paper is an answer to a comment of Masbaum in \cite{Masbaum17} about power subgroups of a half-twist in $\mathcal{M}(0,2n)$. 
The second part treats power subgroups of Dehn twists along all symmetric simple closed curves in the hyperelliptic mapping class group $\Delta(g,0)\subset\mathcal{M}(g,0)$.

We consider the normal closure of $h_i^m$ in $\mathcal{M}(0,2n)$ for $h_i^m$. 
A choice of half-twists is independent of the definition of it because any two half-twists are conjugate in $\mathcal{M}(0,2n)$.
We denote the normal closure by $N_{m}(0,2n)$, or simply $N_m$.
Stylianakis~\cite{Stylianakis17} showed the following theorem by using the Jones representation~\cite{Jones87}.
\begin{THM}[Stylianakis~{\cite{Stylianakis17}}]\label{Stylianakis}
For $2n\geq 6$ and $m\geq 5$, $N_m$ has infinite index in $\mathcal{M}(0,2n)$.
\end{THM}
Masbaum~\cite{Masbaum17} also proved a similar theorem for $2n\geq 4$ and $m\geq 6$ by using a projective representation of $\mathcal{M}(0,2n)$ obtained through the linear skein theory.
\begin{THM}[Masbaum~{\cite{Masbaum17}}]\label{Masbaum}
For $2n\geq 4$ and $m\geq 6$, $N_m$ has infinite index in $\mathcal{M}(0,2n)$.
\end{THM}
Masbaum commented that {\em ``I believe that the remaining case ($2n\geq 6$, $m=5$) of Stylianakis's theorem can also be proved using the skein-theoretic method''}. 
He also expected that the proof of the remaining case probably requires some mathematical software. 
We obtained the following theorem by the skein-theoretic method with no mathematical software.
\setcounter{section}{3}
\setcounter{THM}{5}
\begin{THM}
{\em For any $2n\geq 6$, $m\geq 5$ and $m\neq 6$, 
$N_m$ has infinite index in $\mathcal{M}(0,2n)$}.
\end{THM}
\setcounter{section}{1}
\setcounter{THM}{2}
In the proof of it, 
We make a careful study of the matrix presentation of the projective representation of $\mathcal{M}(0,2n)$ at roots of unity obtained through the Kauffman bracket skein module.

Power subgroups of Dehn twists have been studied not only for a mapping class group of a punctured sphere but also for one of an oriented closed surface of genus $g$. 
For example, Humphries studied the normal closure $\mathcal{N}_m$ of the $m$-th power of a Dehn twist along non-separating simple closed curves (SCC) in $\mathcal{M}(g,0)$.
\begin{THM}[Humphries~\cite{Humphries92}]\label{Humphriesinfty}
For $g=2$, $\mathcal{N}_m$ has infinite index in $\mathcal{M}(g,0)$ if $m\geq 4$.
\end{THM}
We remark that his proof implies that the normal subgroup generated by $m$-th powers of Dehn twists along all SCCs has infinite index in $\mathcal{M}(2,0)$ if $m\geq 4$.
Funar generalize this theorem by using quantum representations of mapping class groups.
\begin{THM}[Funar~\cite{Funar99}]
For any $g\geq 3$, the normal subgroup generated by $m$-th powers of Dehn twists along all SCCs has infinite index in $\mathcal{M}(g,0)$ if $m\not\in\{1,2,3,4,6,8,12\}$.
\end{THM}
We consider the power subgroup of Dehn twists in the hyperelliptic mapping class group $\Delta(g,0)$ of an oriented closed surface of genus $g$ equipped with a hyperelliptic involution $\iota$. 
$\Delta(g,0)$ is the centralizer of the isotopy class of $\iota$ in $\mathcal{M}(g,0)$. 
A typical element in $\Delta(g,0)$ is a Dehn twist along a symmetric ($\iota$-invariant) SCC. 
As is well known, $\mathcal{M}(g,0)=\Delta(g,0)$ if $g=1,2$.
Thus, Theorem~\ref{Humphriesinfty} 
says that the normal closure $\mathcal{N}_m^{\iota}$ of the $m$-th power of a Dehn twist along symmetric non-separating SCC has infinite index in $\Delta(2,0)$ if $m\geq 4$.
Stylianakis obtained a generalization of this theorem as a corollary of Theorem~\ref{Stylianakis}.
\begin{COR}[{\cite[Corollary~1.1]{Stylianakis17}}]\label{Stylianakiscor}
If $g\geq 2$ and $m\geq 4$, then $\mathcal{N}_m^{\iota}$ in $\Delta(g,0)$ has infinite index.
\end{COR}

Our calculations of $\mathcal{M}(0,2n)$ derive a result about the power subgroup of Dehn twists along all symmetric SCCs.
\setcounter{section}{4}
\setcounter{THM}{3}
\begin{COR}
{\em For any $g\geq 2$, the normal subgroup of $m$-th powers of Dehn twists along all symmetric simple closed curves has infinite index in $\Delta(g,0)$ if $m\geq 5$ and $m\neq 6$}.
\end{COR}
\setcounter{section}{1}
\setcounter{THM}{5}

On the other hands, the finiteness of the index of $\mathcal{N}_m$ in $\mathcal{M}(g,0)$ was known as follows.
\begin{THM}[{\cite[Theorem~1, Theorem~2]{Humphries92}}]\label{Humphriesfin}\ 
\begin{itemize}
\item For $g\geq 1$, $\mathcal{N}_2$ has finite index in $\mathcal{M}(g,0)$.
\item For $g=2,3$, $\mathcal{N}_3$ has finite index in $\mathcal{M}(g,0)$.
\end{itemize}
\end{THM}
This theorem implies $\mathcal{N}_m^{\iota}(2,0)$ has finite index in $\Delta(2,0)$ for $m=2,3$ by the Birman-Hilden theorem.

\begin{THM}[Newman~\cite{Newman72}]
$\mathcal{N}_m$ has finite index in $\mathcal{M}(1,0)\cong SL(2,\mathbb{Z})$ for $m\leq 5$, and one has infinite index for $m\geq 6$.
\end{THM}
As a result, we obtaine Table~\ref{map}.

\begin{table}
\begin{tikzpicture}[scale=.5]
\foreach \n in {1,...,22}
\foreach \m in {1,...,10}{
\ifnum \n > 1
\ifnum \m > 4
\fill[cyan] (\n,\m) circle(4pt);
\fi
\fi};
\foreach \m in {5,...,10}{
\fill[cyan] (1,\m) circle(4pt);};
\fill[cyan] (2,4) circle(4pt);
\foreach \n in {1,...,22}{
\fill[black] (\n,1) circle(4pt);};
\foreach \m in {1,...,5}{
\fill[black] (1,\m) circle(4pt);};
\fill[black] (2,2) circle(4pt);
\fill[black] (2,3) circle(4pt);
\foreach \n in {3,...,22}{
\fill[white] (\n,6) circle(4pt);};
\foreach \i in {1,...,22}{
\path (\i,0) node{\i};};
\foreach \j in {1,...,10}{
\path (0,\j) node{\j};};
\foreach \i in {1,...,22}
\foreach \j in {1,...,10}{
\draw (\i,\j) circle(4pt);};
\draw[->] (1,.5) -- (23,.5);
\draw[->] (.5,1) -- (.5,11);
\node at (23,.5) [above]{$g$};
\node at (.5,11) [right]{$m$};
\end{tikzpicture}
\caption{The normal subgroup generated by $m$-th powers of Dehn twists along all symmetric SCCs in $\Delta(g,0)$ has finite index at a black point, and infinite index at a blue point.}
\label{map}
\end{table}

The paper is organized as follows. 
Firstly, we review Masbaum's calculation of a projective representation of $\mathcal{M}(0,2n)$ by using trivalent graph notation of skein elements in Section~\ref{even}. 
In Section~\ref{remainder}, we calculate matrix representation of the projective representation and prove Theorem~\ref{Yuasaeven}. 
In Section~\ref{hyperelliptic}, we consider the power subgroup of Dehn twists along all symmetric SCCs in $\Delta(g,0)$.

\subsection*{Acknowledgment}
The author would like to express his gratitude to Hisaaki Endo for his helpful comments. 
He also would like to thank Yasushi Kasahara for teaching him proof of the remark about Theorem~{1.3}.
\section{Projective representation of $\mathcal{M}(0,2n)$ and the order of $N_{m}(0,2n)$}\label{even}
Let $M$ be a $3$-manifold with framed marked points in its boundary $\partial M$. 
A framed point means an embedding of a ``short'' interval $\left[0,1\right]$ into $\partial M$. 
The Kauffman bracket skein module of the $3$-manifold $M$ is a $\mathbb{C}[q^{\frac{1}{4}},q^{-\frac{1}{4}}]$-vector space spanned by isotopy classes of framed tangles in $M$ whose endpoints lie in the framed marked points in $\partial M$ with the {\em Kauffman bracket skein relations}. 
We remark that the framing of a tangle coincides with the orientation of framed marked points on its edges.
A framed tangle in a $3$-ball $B^3$ in $M$ is described as framed curves with over/under crossings onto the disk $B^3\cap\{z=0\}$. 
The framing of curves on the disk is given by blackboard framing.
The Kauffman bracket skein relations are described as relations of diagrams on a disk in such a way.
\begin{DEF}[The Kauffman bracket skein relation]\ 
\begin{itemize}
\item 
$\,\tikz[baseline=-.6ex]{
\draw[white, double=black, double distance=0.4pt, ultra thick] 
(135:.5) -- (-45:.5);
\draw[white, double=black, double distance=0.4pt, ultra thick] 
(-135:.5) -- (45:.5);
\draw [thick] (0,0) circle [radius=.5];
}\,
=q^{\frac{1}{4}}
\,\tikz[rotate=90, baseline=-.6ex]{
\draw (135:.5) to [out=south east, in=west](0,.2) to [out=east, in=south west](45:.5);
\draw (-135:.5) to [out=north east, in=left](0,-.2) to [out=right, in=north west] (-45:.5);
\draw [thick] (0,0) circle [radius=.5];
}\,
+q^{-\frac{1}{4}}
\,\tikz[baseline=-.6ex]{
\draw (135:.5) to [out=south east, in=west](0,.2) to [out=east, in=south west](45:.5);
\draw (-135:.5) to [out=north east, in=left](0,-.2) to [out=right, in=north west] (-45:.5);
\draw [thick] (0,0) circle [radius=.5];
}\,
$,
\item 
$L\sqcup
\,\tikz[baseline=-.6ex]{
\draw (0,0) circle [radius=.3];
\draw [thick] (0,0) circle [radius=.5];
}\,
=-\left[2\right] L\quad$ for any projected tangle $L$. 
\end{itemize}
The circle described by a thick line is the boundary of $B^3\cap\{z=0\}$. 
For any integer $n$, we define a quantum integer $\left[n\right]=(q^{\frac{n}{2}}-q^{-\frac{n}{2}})/(q^{\frac{1}{2}}-q^{-\frac{1}{2}})$.
\end{DEF}

In this paper, 
we treat the Kauffman bracket skein module $\SK$ of a $3$-ball with $2n$ marked points on its boundary. 
We assume $2n\geq 6$ and take $q\in\mathbb{C}^*$ as a primitive $r$-th root of unity such that $r\geq 3$.
\subsection{Two bases of $\SK$}
Let us consider two types of skein elements of the Kauffman bracket skein module $\SK$ of the $3$-ball with $2n$ marked pints on the boundary. 
They are described by colored trivalent graphs as follows:
\[
\tikz[baseline=-.6ex, scale=0.5]{
\draw[rounded corners, thick] (0,-3) rectangle (10,0);
\draw[rounded corners] (1,0) -- (1,-1.5) -- (2,-1.5);
\draw (2,0) -- (2,-1.5);
\draw (3,0) -- (3,-1.5);
\draw (4,0) -- (4,-1.5);
\draw (5,0) -- (5,-1.5);
\draw (6,0) -- (6,-1.5);
\draw (7,0) -- (7,-1.5);
\draw (8,0) -- (8,-1.5);
\draw[rounded corners] (9,0) -- (9,-1.5) -- (8,-1.5);
\draw (2,-1.5) -- (6,-1.5);
\draw (7,-1.5) -- (8,-1.5);
\draw[fill=cyan] (1,0)  circle (.1);
\draw[fill=cyan] (2,0)  circle (.1);
\draw[fill=cyan] (3,0)  circle (.1);
\draw[fill=cyan] (4,0)  circle (.1);
\draw[fill=cyan] (5,0)  circle (.1);
\draw[fill=cyan] (6,0)  circle (.1);
\draw[fill=cyan] (7,0)  circle (.1);
\draw[fill=cyan] (8,0)  circle (.1);
\draw[fill=cyan] (9,0)  circle (.1);
\draw[fill=black] (2,-1.5)  circle (.1);
\draw[fill=black] (3,-1.5)  circle (.1);
\draw[fill=black] (4,-1.5)  circle (.1);
\draw[fill=black] (5,-1.5)  circle (.1);
\draw[fill=black] (6,-1.5)  circle (.1);
\draw[fill=black] (7,-1.5)  circle (.1);
\draw[fill=black] (8,-1.5)  circle (.1);
\node at (1,0) [above]{$\scriptstyle{1}$};
\node at (1,0) [above left]{$\scriptstyle{a_0=}$};
\node at (2,0) [above]{$\scriptstyle{1}$};
\node at (3,0) [above]{$\scriptstyle{1}$};
\node at (4,0) [above]{$\scriptstyle{1}$};
\node at (5,0) [above]{$\scriptstyle{1}$};
\node at (6,0) [above]{$\scriptstyle{1}$};
\node at (7,0) [above]{$\scriptstyle{1}$};
\node at (8,0) [above]{$\scriptstyle{1}$};
\node at (9,0) [above]{$\scriptstyle{1}$};
\node at (9,0) [above right]{$\scriptstyle{=a_{2n-2}}$};
\node at (2.5,-1.5) [below]{$\scriptstyle{a_1}$};
\node at (3.5,-1.5) [below]{$\scriptstyle{a_2}$};
\node at (4.5,-1.5) [below]{$\scriptstyle{a_3}$};
\node at (5.5,-1.5) [below]{$\scriptstyle{a_4}$};
\node at (6.5,-1.5) {$\scriptstyle{\cdots}$};
\node at (7.5,-1.5) [below]{$\scriptstyle{a_{2n-3}}$};
\node at (0,-1.5) [left]{$\beta_T(a_1,a_2,\dots,a_{2n-3})=$};}
\]
and
\[
\tikz[baseline=-.6ex, scale=0.5]{
\draw[rounded corners, thick] (0,-3) rectangle (10,0);
\draw[rounded corners] (1,0) -- (1,-1.5) -- (2,-1.5);
\draw (2,0) -- (2.5,-.5);
\draw (3,0) -- (2.5,-.5);
\draw (2.5,-1.5) -- (2.5,-.5);
\draw (4,0) -- (4,-1.5);
\draw (5,0) -- (5,-1.5);
\draw (6,0) -- (6,-1.5);
\draw (7,0) -- (7,-1.5);
\draw (8,0) -- (8,-1.5);
\draw[rounded corners] (9,0) -- (9,-1.5) -- (8,-1.5);
\draw (2,-1.5) -- (6,-1.5);
\draw (7,-1.5) -- (8,-1.5);
\draw[fill=cyan] (1,0)  circle (.1);
\draw[fill=cyan] (2,0)  circle (.1);
\draw[fill=cyan] (3,0)  circle (.1);
\draw[fill=cyan] (4,0)  circle (.1);
\draw[fill=cyan] (5,0)  circle (.1);
\draw[fill=cyan] (6,0)  circle (.1);
\draw[fill=cyan] (7,0)  circle (.1);
\draw[fill=cyan] (8,0)  circle (.1);
\draw[fill=cyan] (9,0)  circle (.1);
\draw[fill=black] (2.5,-1.5)  circle (.1);
\draw[fill=black] (2.5,-.5)  circle (.1);
\draw[fill=black] (4,-1.5)  circle (.1);
\draw[fill=black] (5,-1.5)  circle (.1);
\draw[fill=black] (6,-1.5)  circle (.1);
\draw[fill=black] (7,-1.5)  circle (.1);
\draw[fill=black] (8,-1.5)  circle (.1);
\node at (1,0) [above]{$\scriptstyle{1}$};
\node at (1,0) [above left]{$\scriptstyle{a_0=}$};
\node at (2,0) [above]{$\scriptstyle{1}$};
\node at (3,0) [above]{$\scriptstyle{1}$};
\node at (4,0) [above]{$\scriptstyle{1}$};
\node at (5,0) [above]{$\scriptstyle{1}$};
\node at (6,0) [above]{$\scriptstyle{1}$};
\node at (7,0) [above]{$\scriptstyle{1}$};
\node at (8,0) [above]{$\scriptstyle{1}$};
\node at (9,0) [above]{$\scriptstyle{1}$};
\node at (9,0) [above right]{$\scriptstyle{=a_{2n-2}}$};
\node at (2.5,-1) [left]{$\scriptstyle{a_1}$};
\node at (3.5,-1.5) [below]{$\scriptstyle{a_2}$};
\node at (4.5,-1.5) [below]{$\scriptstyle{a_3}$};
\node at (5.5,-1.5) [below]{$\scriptstyle{a_4}$};
\node at (6.5,-1.5) {$\scriptstyle{\cdots}$};
\node at (7.5,-1.5) [below]{$\scriptstyle{a_{2n-3}}$};
\node at (0,-1.5) [left]{$\beta_{T'}(a_1,a_2,\dots,a_{2n-3})=$};}
,\]
where $a_i$ is a non-negative integer for $i=1,2,\dots,2n-3$. 
A trivalent vertex 
$\tikz[baseline=-.6ex, scale=0.5]{
\draw (0,0) -- (1,0);
\draw (-1,1) to[out=east, in=north west] (0,0);
\draw (-1,-1) to[out=east, in=south west] (0,0);
\draw[fill=black] (0,0)  circle (.1);
\node at (-1,1) [above]{$\scriptstyle{b}$};
\node at (-1,-1) [above]{$\scriptstyle{a}$};
\node at (.5,0) [above]{$\scriptstyle{c}$};
}$
means
$\tikz[baseline=-.6ex, scale=0.5]{
\draw (-1,.9) -- +(-.5,0);
\draw (-1,-.9) -- +(-.5,0);
\draw (1,0) -- +(.5,0);
\draw (-1,1) to[out=east, in=west] (1,.1);
\draw (-1,-1) to[out=east, in=west] (1,-.1);
\draw (-1,.8) to[out=east, in=east] (-1,.-.8);
\draw[fill=white] (-1.2,.6) rectangle (-1,1.2);
\draw[fill=white] (-1.2,-.6) rectangle (-1,-1.2);
\draw[fill=white] (1,-.3) rectangle (1.2,.3);
\node at (-1.5,.9)[above]{$\scriptstyle{b}$};
\node at (-1.5,-.9)[above]{$\scriptstyle{a}$};
\node at (1.5,0)[below]{$\scriptstyle{c}$};
\node at (-1,0){$\scriptstyle{k}$};
\node at (0,-.9){$\scriptstyle{j}$};
\node at (0,.9){$\scriptstyle{i}$};},
$
where the white box represents the Jones-Wenzl idempotent and  $i=\frac{b+c-a}{2}$, $j=\frac{c+a-b}{2}$, $k=\frac{a+b-c}{2}$. (See, for example, \cite{Lickorish97} and \cite{KauffmanLins94} in detail.)

\begin{DEF}
A triple of non-negative integers $(a,b,c)$ is {\em admissible} if 
\begin{itemize}
\item $a+b+c$ is an even integer,
\item $a+b-c$, $b+c-a$, and $c+a-b$ are non-negative integers.
\end{itemize}
The admissible coloring $(a,b,c)$ is {\em $q$-admissible} if it satisfies
\begin{itemize}
\item $0\leq a,b,c\leq r-2$ and $a+b+c\leq 2(r-2)$.
\end{itemize}
A colored trivalent graph is {\em admissible} (resp. {\em $q$-admissible}) if, for any trivalent vertex, the triple of adjacent edges has an admissible (resp. $q$-admissible) coloring. 
\end{DEF}
We note that $\mathcal{B}_T=\{\,\beta_T(a_1,a_2,\dots,a_{2n-3}) \mid \text{$q$-admissible}\, \}$ and $\mathcal{B}_{T'}=\{\,\beta_{T'}(a_1,a_2,\dots,a_{2n-3}) \mid \text{$q$-admissible}\,\}$ are dual bases of $\SK$ (see, for example, Lickorish~\cite{Lickorish93B}). 
\begin{RMK}
We take a $3$-ball with $2n$-marked points on its boundary in $S^3$. 
The exterior is also a $3$-ball with the same marked points.
We denote Kauffman bracket modules of them by $\mathsf{S}_q(1^{\otimes 2n})$ and $\bar{\mathsf{S}}_q(1^{\otimes 2n})$.
Then, there is a natural bilinear map $\mathsf{S}_q(1^{\otimes 2n})\times\bar{\mathsf{S}}_q(1^{\otimes 2n})\to\mathbb{C}$.
In this paper, 
we treat the dual space $\SK=\mathsf{S}_q(1^{\otimes 2n})/\{\text{the left kernel}\}$ although Masbaum used $\mathsf{S}_q(1^{\otimes 2n})$ in \cite{Masbaum17}.
\end{RMK}

We observe admissible colorings $(1,1,a_1)$ and $(a_1,a_2,1)$. 
It is easy to see that $(1,1,a_1)$ is admissible if and only if $a_1=0,2$, $(0,a_2,1)$ admissible if and only if $a_2=1$, and $(2,a_2,1)$ admissible if and only if $a_2=1,3$. 
Therefore, 
the $q$-admissible colorings of the bases $\mathcal{B}_T$ and $\mathcal{B}_{T'}$ are fallen into the following three types:
\begin{description}
\item[Type~I] $(a_1,a_2,\dots,a_{2n})$ such that $a_1=0$ and $a_2=1$,
\item[Type~{II}] $(a_1,a_2,\dots,a_{2n})$ such that $a_1=2$ and $a_2=1$,
\item[Type~{III}] $(a_1,a_2,\dots,a_{2n})$ such that $a_1=2$ and $a_2=3$.
\end{description} 
In what follows, we only treat $r\geq 4$ because $q$-admissible colorings of Type~{II} and Type~{III} are empty if $r< 4$.
Then the number of Type~{I} colorings is equal to one of Type~{II}.
Let $k$ and $k'$ be the number of Type~{I} $q$-admissible colorings and Type~{III} $q$-admissible colorings respectively. 
We remark that $k'=0$ if $r=4$.

Let us consider the change of coordinates matrix from $\mathcal{B}_{T}$ to $\mathcal{B}_{T'}$. 
We set up the order of $\mathcal{B}_{T}$ and $\mathcal{B}_{T'}$ as the lexicographic order of colorings $(a_1,a_2,\dots,a_{2n-3})$.
$\mathcal{B}_T$ transforms into $\mathcal{B}_{T'}$ by quantum $6j$-symbols, 
\begin{align*}
\beta_{T}(0,1,a_3,\dots,a_{2n-3})
&=
\begin{Bmatrix}
1&1&0\\
1&1&0
\end{Bmatrix}
\beta_{T'}(0,1,a_3,\dots,a_{2n-3})
+
\begin{Bmatrix}
1&1&2\\
1&1&0
\end{Bmatrix}
\beta_{T'}(2,0,a_3,\dots,a_{2n-3})\\
\beta_{T}(2,1,a_3,\dots,a_{2n-3})
&=\begin{Bmatrix}
1&1&2\\
1&1&2
\end{Bmatrix}
\beta_{T'}(2,1,a_3,\dots,a_{2n-3})
+
\begin{Bmatrix}
1&1&0\\
1&1&2
\end{Bmatrix}
\beta_{T'}(0,1,a_3,\dots,a_{2n-3})\\
\beta_{T}(2,3,a_3,\dots,a_{2n-3})
&=
\begin{Bmatrix}
1&1&2\\
1&3&2
\end{Bmatrix}
\beta_{T'}(2,3,a_3,\dots,a_{2n-3}).
\end{align*} 
The value of the above $6j$-symbols are the following:
\begin{align*}
&\begin{Bmatrix}
1&1&0\\
1&1&0
\end{Bmatrix}
=-\frac{1}{\left[2\right]}, 
&\begin{Bmatrix}
1&1&2\\
1&1&0
\end{Bmatrix}
=1,\\
&\begin{Bmatrix}
1&1&2\\
1&1&2
\end{Bmatrix}
=\frac{1}{\left[2\right]},
&\begin{Bmatrix}
1&1&0\\
1&1&2
\end{Bmatrix}
=\frac{\left[3\right]}{\left[2\right]^2},\\ 
&\begin{Bmatrix}
1&1&0\\
1&1&0
\end{Bmatrix}
=1.
\end{align*}
We rewrite the transformations:
\begin{align*}
\beta_{T}(0,1,a_3,\dots,a_{2n-3})
&=
-\left[2\right]^{-1}\beta_{T'}(0,1,a_3,\dots,a_{n-3})+\beta_{T'}(2,0,a_3,\dots,a_{2n-3})\\
\beta_{T}(2,1,a_3,\dots,a_{2n-3})
&=\left[2\right]^{-1}\beta_{T'}(2,1,a_3,\dots,a_{2n-3})+\left[3\right]\left[2\right]^{-2}\beta_{T'}(0,1,a_3,\dots,a_{2n-3})\\
\beta_{T}(2,3,a_3,\dots,a_{2n-3})
&=\beta_{T'}(2,3,a_3,\dots,a_{2n-3}).
\end{align*} 
It follows from the order of the $q$-admissible colorings that the first $k$ vectors in $\mathcal{B}_{T}$ and $\mathcal{B}_{T'}$ are Type~{I}, the next $k$ vectors are Type~{II} and the $k'$ vectors after the next are Type~{III}. 
Consequently, 
we obtain the following the change of coordinates matrix:
\[
 A=
\begin{bmatrix}
-\left[2\right]^{-1}I_k & \left[3\right]\left[2\right]^{-2}I_k & O\\
I_k & \left[2\right]^{-1}I_k & O\\
O & O & I_{k'}
\end{bmatrix}.
\]
We can also see $A^{-1}=A$ in a straightforward way.

\subsection{Calculation of $\rho$}
There is a natural action of $B_{2n}$ on $\SK$. 
For any standard generator $\sigma_i\in B_{2n}$ ($i=1,2,\dots,2n-1$), the action is given by gluing $\sigma_i^{-1}$ on the top of a skein element in $\SK$. 
It is well-known that $\mathcal{M}(0,2n)$ is a quotient of $B_{2n}$ by the following relators~\cite{Birman74}:
\begin{align*}
R_1(2n)&=\sigma_1\sigma_2\dots\sigma_{2n-1}\sigma_{2n-1}\sigma_{2n-2}\dots\sigma_1\\
R_2(2n)&=(\sigma_1\sigma_2\dots\sigma_{2n-1})^{2n}.
\end{align*}
We remark that the quotient map projects $\sigma_i\in B_{2n}$ onto a half-twist permuting the $i$-th and the $(i+1)$-th punctures.
The action of $R_1(2n)$ and $R_2(2n)$ on $\SK$ is trivial up to scalar multiplication, see~\cite{Masbaum17}.
Hence, we obtain a projective representation $\rho\colon \mathcal{M}(0,2n)\to PGL(\SK)$.
The action of $B_{2n}$ on $\SK$ factor through $\rho$. 
We denote this representation $B_{2n}\to PGL(\SK)$ by the same symbol $\rho$.

Firstly, 
we compute $\rho(\sigma_1^m)$ by using a well-known formula:
\begin{equation}\label{twist}
\tikz[baseline=-.6ex]{
\draw (-1.5,0) -- (-1,0);
\draw (-1,0) to[out=north east, in=west] (-.5,.4);
\draw (-1,0) to[out=south east, in=west] (-.5,-.4);
\draw[white, double=black, double distance=0.4pt, ultra thick] 
(-.5,.4) to[out=east, in=west] (.0,-.4);
\draw[white, double=black, double distance=0.4pt, ultra thick] 
(-.5,-.4) to[out=east, in=west] (.0,.4);
\draw[fill=black] (-1,0) circle (.05);
\node at (-.5,-.5) [left]{$\scriptstyle{c}$};
\node at (-.5,.5) [left]{$\scriptstyle{b}$};
\node at (-1.5,0) [above]{$\scriptstyle{a}$};
}\,
=
(-1)^\frac{a-b-c}{2}q^{-\frac{1}{8}(a(a+2)-b(b+2)-c(c+2))}
\tikz[baseline=-.6ex]{
\draw (-1.5,0) -- (-1,0);
\draw (-1,0) to[out=north east, in=west] (-.5,.4);
\draw (-1,0) to[out=south east, in=west] (-.5,-.4);
\draw[fill=black] (-1,0) circle (.05);
\node at (-.5,-.5) [left]{$\scriptstyle{b}$};
\node at (-.5,.5) [left]{$\scriptstyle{c}$};
\node at (-1.5,0) [above]{$\scriptstyle{a}$};
}\ .
\end{equation}
For a positive integer $m$, the action of $\rho(\sigma_1^m)$ is obtained by gluing a left-handed $m$ half-twists of two strands on the first and second strands of an element of $\SK$.
\[
\sigma_1^m= 
\tikz[baseline=-.6ex, scale=0.5, yscale=0.5]{
\begin{scope}[yshift=2cm]
\draw[white, double=black, double distance=0.4pt, ultra thick] 
(-.5,-1) to[out=north, in=south] (.5,1);
\draw[white, double=black, double distance=0.4pt, ultra thick] 
(.5,-1) to[out=north, in=south] (-.5,1);
\end{scope}
\node at (0,-.5) {$\cdot$};
\node at (0,0) {$\cdot$};
\node at (0,.5) {$\cdot$};
\node at (.5,0) [right]{{\scriptsize $m$ half-twists}};
\begin{scope}[yshift=-2cm]
\draw[white, double=black, double distance=0.4pt, ultra thick] 
(-.5,-1) to[out=north, in=south] (.5,1);
\draw[white, double=black, double distance=0.4pt, ultra thick] 
(.5,-1) to[out=north, in=south] (-.5,1);
\node at (-.5,-1) [below]{$\scriptstyle{1}$};
\node at (.5,-1) [below]{$\scriptstyle{1}$};
\end{scope}
}.\]
It is easy to see that the action of $\sigma_1^m$ on the basis $\beta_T$ is
\begin{align*}
\rho(\sigma_1^m)\beta_T(0,a_1,\dots,a_{2n-3})&=(-1)^mq^{\frac{3m}{4}}\beta_T(0,a_1,\dots,a_{2n-3}),\\
\rho(\sigma_1^m)\beta_T(2,a_1,\dots,a_{2n-3})&=q^{-\frac{m}{4}}\beta_T(2,a_1,\dots,a_{2n-3}).
\end{align*}
Therefore, $\rho(\sigma_1^m)$ is the identity in $PGL(\SK)$ if $(-1)^mq^{\frac{3m}{4}}=q^{-\frac{m}{4}}$.
Of course, 
this calculation coincides with Masbaum's calculation in \cite{Masbaum17}.
\begin{PROP}[Masbaum~{\cite[Prop.~3.1]{Masbaum17}}]\label{Masbaumprop}
If $q\in\mathbb{C}$ satisfies $q^m=(-1)^m$, 
then $\rho(\sigma_i^m)$ is the identity in $PGL(\SK)$.
\end{PROP}

Masbaum proved that there exists a primitive root of unity $q$ satisfying 
\begin{itemize}
\item $\rho(h_i^m)=0$,
\item $\rho(h_1^2h_2^{-2})$ has infinite order in $PGL(\SK)$,
\end{itemize}
for each $2n\geq 4$ and $m\geq 6$. 
It derives the following:
\begin{THM}[Masbaum~{\cite{Masbaum17}}]\label{Masbaum}
For $2n\geq 4$ and $m\geq 6$, $N_m$ has infinite index in $\mathcal{M}(0,2n)$.
\end{THM}
To be more accurate, he did not use the dual space $\SK$ but the Kauffman bracket skein module of the $3$-ball. 
Particularly, he calculated a two-dimensional subspace spanned by tangles ended with four marked points. 
In $\SK$, such subspace are one-dimensional.

We give a brief review of calculations in the proof of Theorem~\ref{Masbaum} by using $\SK$.
\subsection{Calculation of the matrix representation $\rho_T$}
Masbaum showed that $\rho(\sigma_1^2\sigma_2^{-2})$ has infinite order in $PGL(\SK)$ for some primitive root of unity $q$ in \cite{Masbaum17}. 
He actually calculated a matrix representation of $\rho(\sigma_1^2\sigma_2^{-2})$.
Let us follow his argument through a matrix representation $\rho_T$ of $\rho$ with respect to the basis $\mathcal{B}_{T}$ for $r\geq 5$. 
In the same way to the calculation of $\rho(\sigma_i^m)$, 
we can show the following:
\begin{align*}
\rho(\sigma_1)\beta_{T}(0,1,a_3,\dots,a_{2n-3})&=-q^{\frac{3}{4}}\beta_{T}(0,1,a_3,\dots,a_{2n-3}),\\
\rho(\sigma_1)\beta_{T}(2,a_2,a_3,\dots,a_{2n-3})&=q^{-\frac{1}{4}}\beta_{T}(2,a_2,a_3,\dots,a_{2n-3}),\\
\rho(\sigma_2^{-1})\beta_{T'}(0,1,a_3,\dots,a_{2n-3})&=-q^{-\frac{3}{4}}\beta_{T'}(0,1,a_3,\dots,a_{2n-3}),\\
\rho(\sigma_2^{-1})\beta_{T'}(2,a_2,a_3,\dots,a_{2n-3})&=q^{\frac{1}{4}}\beta_{T'}(2,a_2,a_3,\dots,a_{2n-3}).
\end{align*}
Therefore, 
the matrix representation $\rho_T(\sigma_1^s\sigma_2^{-s})$ is
\[
\rho_T(\sigma_1^s\sigma_2^{-s})=
\begin{bmatrix}
(-1)^sq^{\frac{3s}{4}}I_k & O & O\\
O & q^{-\frac{s}{4}}I_k & O\\
O & O & q^{-\frac{s}{4}}I_{k'}
\end{bmatrix}
A
\begin{bmatrix}
(-1)^sq^{-\frac{3s}{4}}I_k & O & O\\
O & q^{\frac{s}{4}}I_k & O\\
O & O & q^{\frac{s}{4}}I_{k'}
\end{bmatrix}
A.
\]
By a straightforward calculation, 
\[
\rho_T(\sigma_1^s\sigma_2^{-s})=
\begin{bmatrix}
(1+(-1)^sq^s\left[3\right])\left[2\right]^{-2}I_k & (-1+(-1)^sq^s)\left[3\right]\left[2\right]^{-3}I_k & O\\
(1-(-1)^sq^{-s})\left[2\right]^{-1}I_k & (1+(-1)^sq^{-s}\left[3\right])\left[2\right]^{-2}I_k & O\\
O & O & I_{k'}
\end{bmatrix},
\]
and $\det \rho_T(\sigma_1^s\sigma_2^{-s})=1$. 
The trace of $\rho_T(\sigma_1^s\sigma_2^{-s})$ can also be computed as $\operatorname{tr}\rho_T(\sigma_1^s\sigma_2^{-s})=f_{s}(q)k+k'$ where
\[
f_s(q)=(-1)^s(q^s+q^{-s})+(-1)^{s+1}\left(\frac{q^{\frac{s}{2}}+(-1)^{s+1}q^{-\frac{s}{2}}}{q^{\frac{1}{2}}+q^{-\frac{1}{2}}}\right)^2.
\]
We apply the argument of \cite{Masbaum99}.
If $\rho_T(\sigma_1^s\sigma_2^{-s})$ has finite order in $PGL_{2k+k'}(\mathbb{C})$, 
then $\left|\operatorname{tr}\rho_T(\sigma_1^s\sigma_2^{-s})\right|\leq 2k+k'$, especially $f_s(q)\leq 2$, for any positive integer $s$ and root of unity $q\in\mathbb{C}$. 
We can derive a contradiction by choosing a primitive $r$-th root of unity such that $f_2(q)>2$ for $r>4$ and $r\not\in\{6,10\}$. 
Thus, we obtained weak version of Lemma~3.4 in \cite{Masbaum17}
\begin{LEM}[Masbaum~\cite{Masbaum17}]\label{Masbaumlem}
Let $\rho_T\colon \mathcal{M}(0,2n)\to PGL_{2k+k}(\mathbb{C})$ be the above representation for $2n\geq 6$. 
Suppose $r\geq 5$ and $r\not\in\{6,10\}$.
Then, 
there exists a primitive $r$-th root of unity $q$ such that $\rho_T(\sigma_1^2\sigma_2^{-2})$ has infinite order in $PGL_{2k+k'}(\mathbb{C})$.
\end{LEM}
We remark that he also showed $\rho(\sigma_1^2\sigma_2^{-2})$ has infinite order for $r=3$. 
In this case, 
the matrix representation he used is non-trivial and all of its eigenvalues are $1$.

For $m\geq 6$, 
we can take a primitive root of unity $q$ such that $\rho_T(\sigma_1^2\sigma_2^{-2})$ has infinite order and $\rho(\sigma_i^m)$ is identity. 
It is a proof of Theorem~\ref{Masbaum} for $2n\geq 6$ and $m\geq 6$.

We remark that his argument works out for $s=1$. 
However, 
we can not prove the infiniteness of the order of $\rho_T(\sigma_1^s\sigma_2^{-s})$ for any other primitive root of unity $q$ whatever integer $s$ we choose. 
Hence, 
we need to consider another element in order to show the remaining case $2n\geq 6$ and $m=5$.

Stylianakis~\cite{Stylianakis17} also showed a similar theorem by representation theory using the Jones representation~\cite{Jones87}.
\begin{THM}[Stylianakis~{\cite{Stylianakis17}}]
For $2n\geq 6$ and $m\geq 5$, $N_m$ has infinite index in $\mathcal{M}(0,2n)$.
\end{THM}
Masbaum commented that ``{\em I believe that the remaining case ($2n\geq 6$, $m=5$) of Stylianakis's theorem can also be proved using the skein-theoretic method}''. 
He also expected that the proof of the remaining case probably requires some mathematical software.

In the next section, 
we will prove the remaining case by hand calculations.
\section{The remaining case}\label{remainder}
In this section, 
we use the basis $\mathcal{B}_T$ and a new one $\mathcal{B}_Y$. 
Let $q$ be a primitive $r$-th root of unity with $r\geq 3$.
$\mathcal{B}_Y$ consists of the following trivalent graph such that a triple of colorings at each trivalent vertex is $q$-admissible.

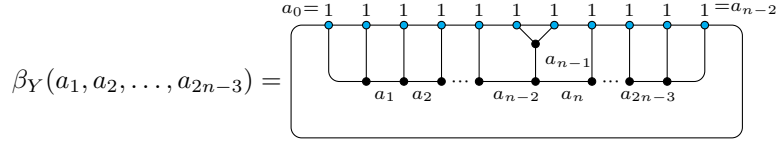
\[
\tikz[baseline=-.6ex, scale=0.5]{
\draw[rounded corners] (0,-3) rectangle (12,0);
\draw[rounded corners] (1,0) -- (1,-1.5) -- (2,-1.5);
\draw (2,0) -- (2,-1.5);
\draw (3,0) -- (3,-1.5);
\draw (4,0) -- (4,-1.5);
\draw (5,0) -- (5,-1.5);
\draw (8,0) -- (8,-1.5);
\draw (9,0) -- (9,-1.5);
\draw (10,0) -- (10,-1.5);
\draw[rounded corners] (11,0) -- (11,-1.5) -- (10,-1.5);
\draw (6,0) -- (6.5,-.5);
\draw (7,0) -- (6.5,-.5);
\draw (6.5,-1.5) -- (6.5,-.5);
\draw (2,-1.5) -- (4,-1.5);
\draw (5,-1.5) -- (8,-1.5);
\draw (9,-1.5) -- (10,-1.5);
\node at (4.5,-1.5) {$\scriptstyle{\cdots}$};
\node at (8.5,-1.5) {$\scriptstyle{\cdots}$};
\draw[fill=cyan] (1,0)  circle (.1);
\draw[fill=cyan] (2,0)  circle (.1);
\draw[fill=cyan] (3,0)  circle (.1);
\draw[fill=cyan] (4,0)  circle (.1);
\draw[fill=cyan] (5,0)  circle (.1);
\draw[fill=cyan] (6,0)  circle (.1);
\draw[fill=cyan] (7,0)  circle (.1);
\draw[fill=cyan] (8,0)  circle (.1);
\draw[fill=cyan] (9,0)  circle (.1);
\draw[fill=cyan] (10,0)  circle (.1);
\draw[fill=cyan] (11,0)  circle (.1);
\draw[fill=black] (2,-1.5)  circle (.1);
\draw[fill=black] (3,-1.5)  circle (.1);
\draw[fill=black] (4,-1.5)  circle (.1);
\draw[fill=black] (5,-1.5)  circle (.1);
\draw[fill=black] (6.5,-1.5)  circle (.1);
\draw[fill=black] (6.5,-.5)  circle (.1);
\draw[fill=black] (8,-1.5)  circle (.1);
\draw[fill=black] (9,-1.5)  circle (.1);
\draw[fill=black] (10,-1.5)  circle (.1);
\node at (1,0) [above]{$\scriptstyle{1}$};
\node at (1,0) [above left]{$\scriptstyle{a_0=}$};
\node at (2,0) [above]{$\scriptstyle{1}$};
\node at (3,0) [above]{$\scriptstyle{1}$};
\node at (4,0) [above]{$\scriptstyle{1}$};
\node at (5,0) [above]{$\scriptstyle{1}$};
\node at (6,0) [above]{$\scriptstyle{1}$};
\node at (7,0) [above]{$\scriptstyle{1}$};
\node at (8,0) [above]{$\scriptstyle{1}$};
\node at (9,0) [above]{$\scriptstyle{1}$};
\node at (10,0) [above]{$\scriptstyle{1}$};
\node at (11,0) [above]{$\scriptstyle{1}$};
\node at (11,0) [above right]{$\scriptstyle{=a_{n-2}}$};
\node at (2.5,-1.5) [below]{$\scriptstyle{a_1}$};
\node at (3.5,-1.5) [below]{$\scriptstyle{a_2}$};
\node at (6,-1.5) [below]{$\scriptstyle{a_{n-2}}$};
\node at (6.5,-1) [right]{$\scriptstyle{a_{n-1}}$};
\node at (7.5,-1.5) [below]{$\scriptstyle{a_{n}}$};
\node at (9.5,-1.5) [below]{$\scriptstyle{a_{2n-3}}$};
\node at (0,-1.5) [left]{$\beta_{Y}(a_1,a_2,\dots,a_{2n-3})=$};}
\]
From the condition of admissible colorings, 
we know that $a_{n-1}$ is $0$ or $2$ for $\beta_Y$.
We define an order on $\mathcal{B}_T$ and $\mathcal{B}_Y$ to give the explicit representation matrix of $\rho$. 
Let us classify the elements of $\mathcal{B}_T$ into four disjoint sets:
\begin{description}
\item[$\operatorname{I}_0(a)$] $\beta_T(a_1,a_2,\dots,a_{2n-3})$ such that $a=a_{n-1}=a_{n-2}+1=a_n-1$,
\item[$\operatorname{I}_2(a)$] $\beta_T(a_1,a_2,\dots,a_{2n-3})$ such that $a=a_{n-1}=a_{n-2}-1=a_n+1$, 
\item[$\operatorname{II}_0(a)$] $\beta_T(a_1,a_2,\dots,a_{2n-3})$ such that $a=a_{n-1}=a_{n-2}+1=a_n+1$, 
\item[$\operatorname{II}_2(a)$] $\beta_T(a_1,a_2,\dots,a_{2n-3})$ such that $a=a_{n-1}=a_{n-2}-1=a_n-1$.
\end{description}

We can describe admissible colorings of $\mathcal{B}_T$ on the lattice by the following manner.
If $\beta_T(a_1,a_{2},\dots,a_{2n-3})$ is admissible, then we mark $(i,a_i)$ with a dot and connect $(i,a_i)$ and $(i+1,a_{i+1})$ by an edge for $i=1,\dots,a_{2n-3}$. 
The elements of $B_{T}$ satisfying $q$-admissibility lie in $\{(x,y)\in\mathbb{Z}_{\geq 0}\times\mathbb{Z}_{\geq 0}\mid y\leq r-2\}$. (See Figure~\ref{latticeT}.)

\begin{figure}
\centering
\begin{tikzpicture}[scale=.5]
\foreach \i in {0,...,4}{
\draw[lightgray] (\i,0) -- (\i,3);};
\foreach \j in {0,...,3}{
\draw[lightgray] (0,\j) -- (4,\j);};
\foreach \j in {0,...,3}{
\path (-.5,\j) node{{\scriptsize \j}};};
\foreach \j in {0,...,4}{
\path (\j,-.5) node{{\scriptsize $a_{\j}$}};};
\draw (0,1) -- (2,3);
\draw (1,0) -- (3,2);
\draw (3,0) -- (4,1);
\draw (4,1) -- (2,3);
\draw (3,0) -- (1,2);
\draw (1,0) -- (0,1);
\fill (0,1) circle(3pt);
\fill (1,0) circle(3pt);
\fill (1,2) circle(3pt);
\fill (2,1) circle(3pt);
\fill (2,3) circle(3pt);
\fill (3,0) circle(3pt);
\fill (3,2) circle(3pt);
\fill (4,1) circle(3pt);
\node[magenta] at (4,1)[right]{{\scriptsize $r=3$}};
\node[magenta] at (4,2)[right]{{\scriptsize $r=4$}};
\node[magenta] at (4,3)[right]{{\scriptsize $r\geq 5$}};
\end{tikzpicture}
\begin{tikzpicture}[scale=.5]
\foreach \i in {0,...,6}{
\draw[lightgray] (\i,0) -- (\i,4);};
\foreach \j in {0,...,4}{
\draw[lightgray] (0,\j) -- (6,\j);};
\foreach \j in {0,...,4}{
\path (-.5,\j) node{{\scriptsize \j}};};
\foreach \i in {0,...,6}{
\path (\i,-.5) node{{\scriptsize $a_{\i}$}};};
\draw (0,1) -- (3,4);
\draw (1,0) -- (4,3);
\draw (3,0) -- (5,2);
\draw (5,0) -- (6,1);
\draw (6,1) -- (3,4);
\draw (5,0) -- (2,3);
\draw (3,0) -- (1,2);
\draw (1,0) -- (0,1);
\fill (0,1) circle(3pt);
\fill (1,0) circle(3pt);
\fill (1,2) circle(3pt);
\fill (2,1) circle(3pt);
\fill (2,3) circle(3pt);
\fill (3,0) circle(3pt);
\fill (3,2) circle(3pt);
\fill (3,4) circle(3pt);
\fill (4,1) circle(3pt);
\fill (4,3) circle(3pt);
\fill (5,0) circle(3pt);
\fill (5,2) circle(3pt);
\fill (6,1) circle(3pt);
\node[magenta] at (6,1)[right]{{\scriptsize $r=3$}};
\node[magenta] at (6,2)[right]{{\scriptsize $r=4$}};
\node[magenta] at (6,3)[right]{{\scriptsize $r=5$}};
\node[magenta] at (6,4)[right]{{\scriptsize $r\geq 6$}};
\end{tikzpicture}

\begin{tikzpicture}[scale=.5]
\draw (-1,-1) -- (1,1);
\fill (-1,-1) circle(3pt);
\fill (0,0) circle(3pt);
\fill (1,1) circle(3pt);
\node at (0,0)[above left]{{\scriptsize $(a_{n-1},a)$}};
\node at (0,-1)[below]{{\small $\operatorname{I}_0(a)$}};
\end{tikzpicture}
\begin{tikzpicture}[scale=.5]
\draw (-1,1) -- (1,-1);
\fill (-1,1) circle(3pt);
\fill (0,0) circle(3pt);
\fill (1,-1) circle(3pt);
\node at (0,0)[below left]{{\scriptsize $(a_{n-1},a)$}};
\node at (0,-1)[below]{{\small $\operatorname{I}_2(a)$}};
\end{tikzpicture}
\begin{tikzpicture}[scale=.5]
\draw (-1,-1) -- (0,0) -- (1,-1);
\fill (-1,-1) circle(3pt);
\fill (0,0) circle(3pt);
\fill (1,-1) circle(3pt);
\node at (0,0)[above]{{\scriptsize $(a_{n-1},a)$}};
\node at (0,-1)[below]{{\small $\operatorname{II}_0(a)$}};
\end{tikzpicture}
\begin{tikzpicture}[scale=.5]
\draw (-1,1) -- (0,0) -- (1,1);
\fill (-1,1) circle(3pt);
\fill (0,0) circle(3pt);
\fill (1,1) circle(3pt);
\node at (0,0)[below]{{\scriptsize $(a_{n-1},a)$}};
\node at (0,-1)[below]{{\small $\operatorname{II}_2(a)$}};
\end{tikzpicture}
\caption{Admissible colorings of $\mathcal{B}_T$ for $n=3,4$}
\label{latticeT}
\end{figure}
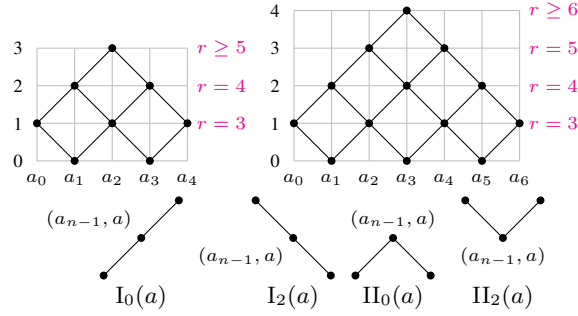

We fix an order on $\operatorname{I}_0(a)$, $\operatorname{I}_2(a)$, $\operatorname{II}_0(a)$, and $\operatorname{II}_2(a)$.
We decide the order of two elements contained in different sets according to the following order of the sets:\\
For a positive odd integer $n$,
\begin{align*}
&\operatorname{I}_0(1)<\operatorname{I}_2(1)<\operatorname{I}_0(3)<\operatorname{I}_2(3)<\dots<\operatorname{I}_0(n-2)<\operatorname{I}_2(n-2)\\
&\quad<\operatorname{II}_0(1)<\operatorname{II}_2(1)<\operatorname{II}_0(3)<\operatorname{II}_2(3)<\dots<\operatorname{II}_0(n-2)<\operatorname{II}_2(n-2)<\operatorname{II}_0(n).
\end{align*}
For a positive even integer $n$,
\begin{align*}
&\operatorname{I}_0(2)<\operatorname{I}_2(2)<\operatorname{I}_0(4)<\operatorname{I}_2(4)<\dots<\operatorname{I}_0(n-2)<\operatorname{I}_2(n-2)\\
&<\operatorname{II}_2(0)<\operatorname{II}_0(2)<\operatorname{II}_2(2)<\operatorname{II}_0(4)<\operatorname{II}_2(4)<\dots<\operatorname{II}_0(n-2)<\operatorname{II}_2(n-2)<\operatorname{II}_0(n).
\end{align*}
We remark that the $q$-admissibility implies that 
\begin{itemize}
\item $\operatorname{I}_{i}(k-1)=\emptyset$ and $\operatorname{II}_{i}(k-1)=\emptyset$ for $k\geq r$ and $i=0,2$ when $n+r$ is odd,
\item $\operatorname{I}_{i}(k-2)=\emptyset$, $\operatorname{II}_{2}(r-2)=\emptyset$ and $\operatorname{II}_{i}(k)=\emptyset$ for $k\geq r$ and $i=0,2$ when $n+r$ is even. 
\end{itemize}

In the same way, 
the basis $\mathcal{B}_Y$ decompose into
\begin{description}
\item[$\operatorname{I}'_0(a)$] $\beta_Y(a_1,a_2,\dots,a_{2n-3})$ such that $a=a_{n-2}+1=a_{n}-1$ (then $a_{n-1}=2$),
\item[$\operatorname{I}'_2(a)$] $\beta_Y(a_1,a_2,\dots,a_{2n-3})$ such that $a=a_{n-2}-1=a_{n}+1$ (then $a_{n-1}=2$), 
\item[$\operatorname{II}'_0(a)$] $\beta_Y(a_1,a_2,\dots,a_{2n-3})$ such that $a=a_{n-2}+1=a_{n}+1$ and $a_{n-1}=0$, 
\item[$\operatorname{II}'_2(a)$] $\beta_Y(a_1,a_2,\dots,a_{2n-3})$ such that $a=a_{n-2}+1=a_{n}+1$ and $a_{n-1}=2$. 
\end{description}
The admissible elements of $\mathcal{B}_Y$ are also described as diagrams in the same way to $\mathcal{B}_T$. (See Figure~\ref{latticeY}.) 
We remark that $\operatorname{II}'_2(1)$ always contains no admissible elements.

\begin{figure}
\centering
\begin{tikzpicture}[scale=.5]
\foreach \i in {0,...,4}{
\draw[lightgray] (\i,0) -- (\i,3);};
\foreach \j in {0,...,3}{
\draw[lightgray] (0,\j) -- (4,\j);};
\foreach \j in {0,...,3}{
\path (-.5,\j) node{{\scriptsize \j}};};
\foreach \j in {0,...,4}{
\path (\j,-.5) node{{\scriptsize $a_{\j}$}};};
\draw (0,1) -- (1,2);
\draw (3,0) -- (4,1);
\draw (4,1) -- (3,2);
\draw (1,0) -- (0,1);
\draw (1,0) -- (3,0);
\draw (1,2) -- (2,0) -- (3,2);
\fill (0,1) circle(3pt);
\fill (1,0) circle(3pt);
\fill (1,2) circle(3pt);
\fill (2,0) circle(3pt);
\fill (3,0) circle(3pt);
\fill (3,2) circle(3pt);
\fill (4,1) circle(3pt);
\node at (2,-.5)[below]{{\small $a_2=0$}};
\end{tikzpicture}
\begin{tikzpicture}[scale=.5]
\foreach \i in {0,...,4}{
\draw[lightgray] (\i,0) -- (\i,3);};
\foreach \j in {0,...,3}{
\draw[lightgray] (0,\j) -- (4,\j);};
\foreach \j in {0,...,3}{
\path (-.5,\j) node{{\scriptsize \j}};};
\foreach \j in {0,...,4}{
\path (\j,-.5) node{{\scriptsize $a_{\j}$}};};
\draw (0,1) -- (1,2);
\draw (4,1) -- (3,2);
\draw (1,2) -- (3,2);
\fill (0,1) circle(3pt);
\fill (1,2) circle(3pt);
\fill (2,2) circle(3pt);
\fill (3,2) circle(3pt);
\fill (4,1) circle(3pt);
\node at (2,-.5)[below]{{\small $a_2=2$}};
\end{tikzpicture}

\begin{tikzpicture}[scale=.5]
\foreach \i in {0,...,6}{
\draw[lightgray] (\i,0) -- (\i,4);};
\foreach \j in {0,...,4}{
\draw[lightgray] (0,\j) -- (6,\j);};
\foreach \j in {0,...,4}{
\path (-.5,\j) node{{\scriptsize \j}};};
\foreach \i in {0,...,6}{
\path (\i,-.5) node{{\scriptsize $a_{\i}$}};};
\draw (0,1) -- (2,3);
\draw (1,0) -- (2,1);
\draw (3,0) -- (5,2);
\draw (5,0) -- (6,1);
\draw (6,1) -- (4,3);
\draw (5,0) -- (4,1);
\draw (3,0) -- (1,2);
\draw (1,0) -- (0,1);
\draw (2,3) -- (3,0) -- (4,3);
\fill (0,1) circle(3pt);
\fill (1,0) circle(3pt);
\fill (1,2) circle(3pt);
\fill (2,1) circle(3pt);
\fill (2,3) circle(3pt);
\fill (3,0) circle(3pt);
\fill (4,1) circle(3pt);
\fill (4,3) circle(3pt);
\fill (5,0) circle(3pt);
\fill (5,2) circle(3pt);
\fill (6,1) circle(3pt);
\node at (3,-.5)[below]{{\small $a_3=0$}};
\end{tikzpicture}
\begin{tikzpicture}[scale=.5]
\foreach \i in {0,...,6}{
\draw[lightgray] (\i,0) -- (\i,4);};
\foreach \j in {0,...,4}{
\draw[lightgray] (0,\j) -- (6,\j);};
\foreach \j in {0,...,4}{
\path (-.5,\j) node{{\scriptsize \j}};};
\foreach \i in {0,...,6}{
\path (\i,-.5) node{{\scriptsize $a_{\i}$}};};
\draw (0,1) -- (2,3);
\draw (1,0) -- (4,3);
\draw (4,1) -- (5,2);
\draw (5,0) -- (6,1);
\draw (6,1) -- (4,3);
\draw (5,0) -- (2,3);
\draw (2,1) -- (1,2);
\draw (1,0) -- (0,1);
\fill (0,1) circle(3pt);
\fill (1,0) circle(3pt);
\fill (1,2) circle(3pt);
\fill (2,1) circle(3pt);
\fill (2,3) circle(3pt);
\fill (3,2) circle(3pt);
\fill (4,1) circle(3pt);
\fill (4,3) circle(3pt);
\fill (5,0) circle(3pt);
\fill (5,2) circle(3pt);
\fill (6,1) circle(3pt);
\node at (3,-.5)[below]{{\small $a_3=2$}};
\end{tikzpicture}
\caption{Admissible colorings of $\mathcal{B}_Y$ for $n=3,4$}
\label{latticeY}
\end{figure}
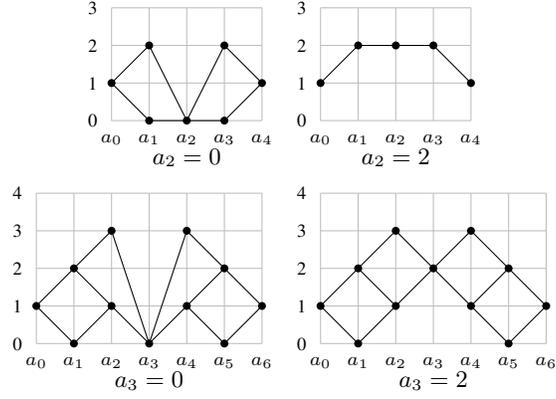

For each set, we fix an order on it. 
If two elements are contained in different sets, 
then we decide the order of them according to the following order on the sets:\\
For positive odd integer $n$,
\begin{align*}
&\operatorname{I}'_0(1)<\operatorname{I}'_2(1)<\operatorname{I}'_0(3)<\operatorname{I}'_2(3)<\dots<\operatorname{I}'_0(n-2)<\operatorname{I}'_2(n-2)\\
&\quad<\operatorname{II}'_0(1)<\operatorname{II}'_0(3)<\operatorname{II}'_2(3)<\operatorname{II}'_0(5)<\operatorname{II}'_2(5)<\dots<\operatorname{II}'_0(n)<\operatorname{II}'_2(n).
\end{align*}
For a positive even integer $n$,
\begin{align*}
&\operatorname{I}'_0(2)<\operatorname{I}'_2(2)<\operatorname{I}'_0(4)<\operatorname{I}'_2(4)<\dots<\operatorname{I}'_0(n-2)<\operatorname{I}'_2(n-2)\\
&\quad<\operatorname{II}'_0(2)<\operatorname{II}'_2(2)<\operatorname{II}'_0(4)<\operatorname{II}'_2(4)<\dots<\operatorname{II}'_0(n)<\operatorname{II}'_2(n).
\end{align*}

We denote the number of $q$-admissible elements of these sets by $k_i(a)=\left|\operatorname{I}_i(a)\right|$, 
$l_i(a)=\left|\operatorname{II}_i(a)\right|$, 
$k'_i(a)=\left|\operatorname{I}'_i(a)\right|$, and
$l'_i(a)=\left|\operatorname{II}'_i(a)\right|$ for $i=0,2$.
We remark that $k_0(a)=k_2(a)=k'_0(a)=k'_2(a)$ and $l_0(a)=l_2(a-2)=l'_0(a)=l'_2(a)$.

\subsection{The case with odd $n\geq3$}
Let $n\geq 3$ be an odd integer.
We will calculate the matrix representation $M=\rho_T((\sigma_1\sigma_2\dots\sigma_{n-1})^{n}\sigma_n(\sigma_1\sigma_2\dots\sigma_{n-1})^{-n}\sigma_n^{-1})$ with respect to $\mathcal{B}_T$ in a similar way to Section~\ref{even}. 
Let us determine the change of coordinates matrix from $\mathcal{B}_T$ to $\mathcal{B}_Y$. 
At the beginning, 
we consider about $q$-admissibility of $(a_{n-2},a_{n-1},a_{n})$ of $\mathcal{B}_T$ and $(a_{n-2},i,a_{n})$ of $\mathcal{B}_Y$ ($i=0,2$). 
\begin{RMK}
In what follows, 
we discuss in the case of $r-4\leq n-2$ when $r$ is odd and $r-3\leq n-2$ when $r$ is even. 
Otherwise we can carry the same argument by replacing $r-4$ and $r-3$ by $n-2$.
\end{RMK}
\begin{LEM}
\begin{itemize}
\item If $(a_{n-2},a_{n-1},a_{n})=(a-1,a,a+1)$ is $q$-admissible, then $(a_{n-2},2,a_{n})$ is $q$-admissible.
\item If $(a_{n-2},a_{n-1},a_{n})=(a+1,a,a-1)$ is $q$-admissible, then $(a_{n-2},2,a_{n})$ is $q$-admissible.
\item If $(a_{n-2},a_{n-1},a_{n})=(a-1,a,a-1)$ is $q$-admissible, then $(a_{n-2},0,a_{n})$ is $q$-admissible, and $(a_{n-2},2,a_{n})$ is $q$-admissible except for $a=1$.
\item If $(a_{n-2},a_{n-1},a_{n})=(a+1,a,a+1)$ is $q$-admissible, then $(a_{n-2},0,a_{n})$ is $q$-admissible, and $(a_{n-2},2,a_{n})$ is $q$-admissible except for $a=r-3$.
\end{itemize}
\end{LEM}
\begin{proof}
We remark that $(a_{n-2},0,a_{n})$ is not admissible when $(a_{n-2},a_{n-1},a_{n})$ is admissible coloring of $\operatorname{I}_i(a)$ for $i=0,2$. 
$(a_{n-2},a_{n-1},a_{n})=(a-1,a,a+1)$ is $q$-admissible if and only if $a\leq r-3$. 
Then $a_{n-1}+2+a_{n}=2a+2\leq 2(a-2)$. In other cases, we can prove in the same way. 
\end{proof}
For $q$-admissible elements in $\mathcal{B}_T$, 
we transforms them into a sum of $q$-admissible elements in $\mathcal{B}_Y$ by calculating quantum $6$j-symbols. 
we obtain
\begin{align*}
&\beta_T(a_1,\dots,a_{n-3},a-1,a,a+1,a_{n+1},\dots,a_{2n-3})\\
&\qquad=\beta_Y(a_1,\dots,a_{n-3},a-1,2,a+1,a_{n+1},\dots,a_{2n-3}).
\end{align*}
\begin{align*}
&\beta_T(a_1,\dots,a_{n-3},a+1,a,a-1,a_{n+1},\dots,a_{2n-3})\\
&\qquad=\beta_Y(a_1,\dots,a_{n-3},a+1,2,a-1,a_{n+1},\dots,a_{2n-3}).
\end{align*}
\begin{align*}
&\beta_T(a_1,\dots,a_{n-3},a-1,a,a-1,a_{n+1},\dots,a_{2n-3})\\
&\quad=\frac{\left[a+1\right]}{\left[2\right]\left[a\right]}\beta_Y(a_1,\dots,a_{n-3},a-1,0,a-1,a_{n+1},\dots,a_{2n-3})\\
&\qquad+\frac{\left[a-1\right]}{\left[a\right]}\beta_Y(a_1,\dots,a_{n-3},a-1,2,a-1,a_{n+1},\dots,a_{2n-3}).
\end{align*}
For $a< r-3$,
\begin{align*}
&\beta_T(a_1,\dots,a_{n-3},a+1,a,a+1,a_{n+1},\dots,a_{2n-3})\\
&\quad=-\left[2\right]^{-1}\beta_Y(a_1,\dots,a_{n-3},a+1,0,a+1,a_{n+1},\dots,a_{2n-3})\\
&\qquad+\beta_Y(a_1,\dots,a_{n-3},a+1,2,a+1,a_{n+1},\dots,a_{2n-3}).
\end{align*}
For $a=r-3$,
\begin{align*}
&\beta_T(a_1,\dots,a_{n-3},a+1,a,a+1,a_{n+1},\dots,a_{2n-3})\\
&\quad=-\left[2\right]^{-1}\beta_Y(a_1,\dots,a_{n-3},a+1,0,a+1,a_{n+1},\dots,a_{2n-3}).
\end{align*}
We can see that $\operatorname{I}_i(a)$ transforms into $\operatorname{I}'_i(a)$ by identity for $i=0,2$,
$\operatorname{II}_0(1)$ into $\operatorname{II}'_0(1)$ by identity, $\operatorname{II}_2(a)\cup\operatorname{II}_0(a+2)$ into $\operatorname{II}'_0(a+2)\cup\operatorname{II}'_2(a+2)$ by a matrix $X(a)$:\\
For $a<r-3$,
\[
 X(a)
=\begin{bmatrix}
-\left[2\right]^{-1}I_{l_2(a)}&\frac{\left[a+3\right]}{\left[2\right]\left[a+2\right]}I_{l_2(a)}\\
I_{l_2(a)}&\frac{\left[a+1\right]}{\left[a+2\right]}I_{l_2(a)}
\end{bmatrix},
\quad 
X(a)^{-1}
=\begin{bmatrix}
-\frac{\left[a+1\right]}{\left[a+2\right]}I_{l_2(a)}&\frac{\left[a+3\right]}{\left[2\right]\left[a+2\right]}I_{l_2(a)}\\
I_{l_2(a)}&\left[2\right]^{-1}I_{l_2(a)}
\end{bmatrix}.
\]
and,
\[
 X(r-3)
=-\left[2\right]^{-1}I_{l_2(r-3)}.
\]
Hence, 
the change of coordinates matrix is 
\[
 I_{k_{r-4}}\oplus I_{l_1(1)}\oplus X(1)\oplus X(3)\oplus\dots\oplus X(r-4) \quad \text{($r$ is odd)},
\]
or
\[
 I_{k_{r-3}}\oplus I_{l_1(1)}\oplus X(1)\oplus X(3)\oplus\dots\oplus X(r-5)\oplus X(r-3) \quad \text{($r$ is even)},
\]
where $k_m=2(k_1(1)+k_1(3)+\dots+k_1(m))$ for odd integer $m$.
It is easy to see that the matrix representation of $\sigma_{n}^j$ with respect to $\mathcal{B}_Y$ is 
\begin{equation*}
\rho_{Y}(\sigma_n^j)=
\begin{cases}
q^{-\frac{j}{4}}I_{k_{r-4}}\oplus (-1)^jq^{\frac{3j}{4}}I_{l_1(1)}\oplus Y(3)^j\oplus Y(5)^j\oplus\dots\oplus Y(r-2)^j &\text{($r$ is odd)},\\
q^{-\frac{j}{4}}I_{k_{r-3}}\oplus (-1)^jq^{\frac{3j}{4}}I_{l_1(1)}\oplus Y(3)^j\oplus Y(5)^j\oplus\dots\oplus Y(r-1)^j&\text{($r$ is even)},
\end{cases}
\end{equation*}
where
$Y(a)=-q^{\frac{3}{4}}I_{l'_2(a)}\oplus q^{-\frac{1}{4}}I_{l'_2(a)}$ for $a<r-2$ and $Y(r-1)=-q^{\frac{3}{4}}I_{l'_2(r-1)}$.
Therefore, 
\begin{align*}
\rho_T(\sigma_n^j)&=q^{-\frac{j}{4}}I_{k_{r-4}}\oplus (-1)^jq^{\frac{3j}{4}}I_{l_1(1)}\oplus X(1)^{-1}Y(3)^jX(1)\oplus X(3)^{-1}Y(5)^jX(3)\\
&\quad\oplus\dots\oplus X(r-4)^{-1}Y(r-2)^jX(r-4),
\end{align*}
if $r$ is odd.
\begin{align*}
\rho_T(\sigma_n^j)&=q^{-\frac{j}{4}}I_{k_{r-3}}\oplus (-1)^jq^{\frac{3j}{4}}I_{l_1(1)}\oplus X(1)^{-1}Y(3)^jX(1)\oplus X(3)^{-1}Y(5)^jX(3)\\
&\quad\oplus\dots\oplus X(r-5)^{-1}Y(r-3)^jX(r-5)\oplus (-1)^jq^{\frac{3j}{4}}I_{l_2(r-3)},
\end{align*}
if $r$ is even.
The braid $(\sigma_1\sigma_2\dots\sigma_{m-1})^m$ acts on the strands ranging from the first to the $m$-th as a left-hand full twist for $m=2,3,\dots,2n$. 
This action on $\beta_T(a_1,a_2,\dots,a_{2n-3})$ is easily calculated by using isotopy invariance of the skein element. 
In fact,
\begin{align}
\rho((\sigma_1\sigma_2\dots\sigma_{m-1})^m)(\beta_T(a_1,a_2,\dots,a_{2n-3}))
&=(-q^{\frac{3}{4}})^m(-1)^{a_{m-1}}q^{-\frac{a_{m-1}^2+2a_{m-1}}{4}}
\beta_T(a_1,a_2,\dots,a_{2n-3})\label{fulltwist}\\
&=(-1)^{m+a_{m-1}}q^{\frac{3m}{4}}q^{-\frac{a_{m-1}^2+2a_{m-1}}{4}}\beta_T(a_1,a_2,\dots,a_{2n-3})\notag
\end{align}
Thus,
\begin{align*}
\rho_T((\sigma_1\sigma_2\dots\sigma_{n-1})^{n})
&=f_1(q)I_{2k_1(1)}\oplus f_3(q)I_{2k_1(3)}\oplus\dots\oplus f_{r-4}(q)I_{2k_1(r-4)}\\
&\quad\oplus f_{1}(q)I_{l_1(1)}\oplus Z(1)\oplus Z(3)\oplus\dots\oplus Z(r-4),
\end{align*}
if $r$ is odd.
\begin{align*}
\rho_T((\sigma_1\sigma_2\dots\sigma_{n-1})^{n})
&=f_1(q)I_{2k_1(1)}\oplus f_3(q)I_{2k_1(3)}\oplus\dots\oplus f_{r-3}(q)I_{2k_1(r-3)}\\
&\quad\oplus f_{1}(q)I_{l_1(1)}\oplus Z(1)\oplus Z(3)\oplus\dots\oplus Z(r-5)\oplus f_{r-3}(q)I_{l_2(r-3)},
\end{align*}
if $r$ is even.
In the above, $f_a(q)=q^{\frac{3n}{4}}q^{-\frac{a^2+2a}{4}}$ and 
$
Z(a)=
f_{a}(q)I_{l_2(a)}\oplus f_{a+2}(q)I_{l_0(a+2)}.
$
Consequently, we denote the representation matrix of $\rho_T((\sigma_1\sigma_2\dots\sigma_{n-1})^{n}\sigma_n(\sigma_1\sigma_2\dots\sigma_{n-1})^{-n}\sigma_n^{-1})$ by $M$,
\begin{align*}
M=
\begin{cases}
I_{k_{r-4}}\oplus I_{l_1(1)}\oplus M(1)\oplus M(3)\oplus\dots\oplus M(r-4) & \text{if $r$ is odd},\\ 
I_{k_{r-3}}\oplus I_{l_1(1)}\oplus M(1)\oplus M(3)\oplus\dots\oplus M(r-5)\oplus I_{l_2(r-3)} & \text{if $r$ is even}, 
\end{cases}
\end{align*}
where
{\small
\begin{align*}
&M(a)=Z(a)X(a)^{-1}Y(a)X(a)Z(a)^{-1}X(a)^{-1}Y(a)^{-1}X(a)\\
&=
\begin{bmatrix}
\left(1-(1-f_a(q)f_{a+2}(q^{-1}))
\frac{\left[a+1\right]\left[a+3\right]}{\left[a+2\right]^2}\right)I_{l_2(a)}
&(1-f_a(q)f_{a+2}(q^{-1}))
\frac{(-q^{\frac{1}{2}}\left[a+1\right]+q^{-\frac{1}{2}}\left[a+3\right])\left[a+1\right]\left[a+3\right]}{\left[2\right]\left[a+2\right]^3}I_{l_2(a)}\\
(1-f_a(q^{-1})f_{a+2}(q))
\frac{q^{-\frac{1}{2}}\left[a+1\right]-q^{\frac{1}{2}}\left[a+3\right]}{\left[2\right]\left[a+2\right]}I_{l_2(a)}
&\left(1-(1-f_a(q^{-1})f_{a+2}(q))
\frac{\left[a+1\right]\left[a+3\right]}{\left[a+2\right]^2}\right)I_{l_2(a)}
\end{bmatrix}.
\end{align*}}
Let us compute the trace of $M(a)$.
\begin{align*}
\operatorname{tr}M(a)
&=\left(2-(2-f_a(q)f_{a+2}(q^{-1})-f_a(q^{-1})f_{a+2}(q))\frac{\left[a+1\right]\left[a+3\right]}{\left[a+2\right]^2}\right)l_2(a)\\
&=\left(2-(2-q^{a+2}-q^{-a-2})\frac{\left[a+1\right]\left[a+3\right]}{\left[a+2\right]^2}\right)l_2(a)\\
&=\left(2+(q^{\frac{a+2}{2}}-q^{-\frac{a+2}{2}})^2\frac{\left[a+1\right]\left[a+3\right]}{\left[a+2\right]^2}\right)l_2(a)\\
&=2l_2(a)+(q^{\frac{a+1}{2}}-q^{-\frac{a+1}{2}})(q^{\frac{a+3}{2}}-q^{-\frac{a+3}{2}})l_2(a)
\end{align*}

We show that 
$\rho((\sigma_1\sigma_2\dots\sigma_{n-1})^{n}\sigma_n(\sigma_1\sigma_2\dots\sigma_{n-1})^{-n}\sigma_n^{-1})$ 
has infinite order in $PGL(\SK)$ for some primitive $r$-th roots of unity $q$. 
Let us consider when $M(a)$ has infinite order by the same argument in \cite{Masbaum99, Masbaum17}. 
When $M(a)$ has finite order in $PGL_{2l_2(a)}(\mathbb{C})$, all eigenvalues of $M(a)$ should be roots of unity and $\left|\operatorname{tr}M(a)\right|\leq 2l_2(a)$ for any $a$. 
We can prove $M(a)$ has infinite order if $\operatorname{tr}M(a)>2l_2(a)$ by contradiction. 
We already calculated $\operatorname{tr}M(a)$. This condition is
\[
(q^{\frac{a+1}{2}}-q^{-\frac{a+1}{2}})(q^{\frac{a+3}{2}}-q^{-\frac{a+3}{2}})>0.
\]
$\operatorname{II}_{2}(1)$ and $\operatorname{II}_0(3)$ are non-empty sets if $r\geq 5$. 
Therefore, we pick out a primitive roots of unity of order $r\geq 5$ which satisfies $\operatorname{tr}M(1)>2l_2(1)$, that is,  $(q^{\frac{a+1}{2}}-q^{-\frac{a+1}{2}})(q^{\frac{a+3}{2}}-q^{-\frac{a+3}{2}})>0$. 
Then, the angle of $q=\exp(i\theta)$ is $\pi/2<\theta<\pi$ or $\pi<\theta<3\pi/2$.
It is easily seen that such primitive root of unity $q$ exists for any $r\geq 5$ except for $r= 6$. 
For example, $q=\exp(3\pi/5)$ ($r=10$) satisfies it.
In the case of $r=6$, we calculate $M(1)$ and its eigenvalues for $q=\pi/3$.
Therefore, we obtain the following.
\begin{LEM}\label{Yuasalemodd}
Let $n$ be odd integer such that $2n\geq 6$. 
Then, there exists a primitive $r$-th root of unity $q$ such that $M$ has infinite order in $PGL(\mathbb{C})$ for any $r\geq 5$ and $r\neq 6$.
\end{LEM}

\subsection{The case with even $n\geq4$}
Let $n\geq 4$ be an even integer. 
We can obtain the following as is the case for an odd $n$.
The change of coordinates matrix is 
\[
 I_{k_{r-4}}\oplus X(0)\oplus X(2)\oplus\dots\oplus X(r-4) \quad \text{($r$ is even)},
\]
\[
 I_{k_{r-3}}\oplus X(0)\oplus X(2)\oplus\dots\oplus X(r-5)\oplus X(r-3) \quad \text{($r$ is odd)},
\]
where $k_{m}=2(k_1(0)+k_1(2)+\dots+k_1(m))$ for positive even integer $m$.
The matrix representation of $\sigma_{n}^j$ with respect to $\mathcal{B}_Y$ is 
\begin{equation*}
\rho_{Y}(\sigma_n^j)=
\begin{cases}
q^{-\frac{j}{4}}I_{k_{r-4}}\oplus Y(2)^j\oplus Y(4)^j\oplus\dots\oplus Y(r-2)^j &\text{($r$ is even)},\\
q^{-\frac{j}{4}}I_{k_{r-3}}\oplus Y(2)^j\oplus Y(4)^j\oplus\dots\oplus Y(r-1)^j&\text{($r$ is odd)}.
\end{cases}
\end{equation*}
If $r$ is even, then
\begin{align*}
\rho_T(\sigma_n^j)&=q^{-\frac{j}{4}}I_{k_{r-4}}\oplus X(0)^{-1}Y(2)^jX(0)\oplus X(2)^{-1}Y(4)^jX(2)\\
&\quad\oplus\dots\oplus X(r-4)^{-1}Y(r-2)^jX(r-4),\\
\rho_T((\sigma_1\sigma_2\dots\sigma_{n-1})^{n})
&=f_2(q)I_{2k_1(2)}\oplus f_4(q)I_{2k_1(4)}\oplus\dots\oplus f_{r-4}(q)I_{2k_1(r-4)}\\
&\quad\oplus Z(0)\oplus Z(2)\oplus\dots\oplus Z(r-4).
\end{align*}
If $r$ is odd, then
\begin{align*}
\rho_T(\sigma_n^j)&=q^{-\frac{j}{4}}I_{k_{r-3}}\oplus X(0)^{-1}Y(2)^jX(0)\oplus X(2)^{-1}Y(4)^jX(2)\\
&\quad\oplus\dots\oplus X(r-5)^{-1}Y(r-3)^jX(r-5)\oplus (-1)^jq^{\frac{3j}{4}}I_{l_2(r-3)},\\
\rho_T((\sigma_1\sigma_2\dots\sigma_{n-1})^{n})
&=f_1(q)I_{2k_1(1)}\oplus f_3(q)I_{2k_1(3)}\oplus\dots\oplus f_{r-3}(q)I_{2k_1(r-3)}\\
&\quad\oplus Z(0)\oplus Z(2)\oplus\dots\oplus Z(r-5)\oplus f_{r-3}(q)I_{l_2(r-3)},
\end{align*}
As a result, we obtain the representation matrix:
\begin{align*}
M=
\begin{cases}
I_{k_{r-4}}\oplus M(0)\oplus M(2)\oplus\dots\oplus M(r-4) & \text{if $r$ is even},\\ 
I_{k_{r-3}}\oplus M(0)\oplus M(2)\oplus\dots\oplus M(r-5)\oplus I_{l_2(r-3)} & \text{if $r$ is odd}. 
\end{cases}
\end{align*}

We show that $M_n(a)$ has infinite order by using the same way as before.
$M_n(0)$ exists if $r\geq 4$, and $\operatorname{tr}M_n(0)> 2l_2(0)$ if $(q^{\frac{1}{2}}-q^{-\frac{1}{2}})(q^{\frac{3}{2}}-q^{-\frac{3}{2}})>0$.
It can be seen that there exists a primitive $r$-th root of unity satisfying it for $r\geq 5$ and $r\not\in\{6,10\}$. 
This result is included in Lemma~\ref{Masbaumlem}.
We next consider when $r\geq 6$. 
Then, we can consider $M_n(2)$ because $l_2(2)>0$.
$\operatorname{tr}M(2)> 2l_2(2)$ is equivalent to $(q^{\frac{3}{2}}-q^{-\frac{3}{2}})(q^{\frac{5}{2}}-q^{-\frac{5}{2}})>0$. 
A tenth root of unity $q=\exp(3\pi/5)$ satisfies this condition. 
\begin{LEM}\label{Yuasalemeven}
Let $n$ be even integer such that $2n\geq 8$. 
Then, there exists a primitive $r$-th root of unity $q$ such that $M$ has infinite order in $PGL(\mathbb{C})$ for any $r\geq 5$ and $r\neq 6$.
\end{LEM}

\begin{RMK}
In the case of $r=4$, the matrix $M(0)$ has the order $2$.
In the case of $r=6$, the matrix $M(1)$ has order $3$, the matrices $M(0)$ and $M(2)$ have order $2$.
\end{RMK}
In fact,
we can simplify $M(a)$ as
\begin{align*}
\begin{bmatrix}
\left(1+2iq^{\frac{q+2}{2}}\frac{\sin(\frac{a+1}{2}\theta)\sin(\frac{a+3}{2}\theta)}{\sin(\frac{a+2}{2}\theta)}\right)I_{l_2(1)}
&-2i\sin(\frac{\theta}{2})\frac{\sin(\frac{a+1}{2}\theta)\sin(\frac{a+3}{2}\theta)}{\sin^2(\frac{a+2}{2}\theta)}I_{l_2(1)}\\
-2i\sin(\frac{\theta}{2})I_{l_2(1)}
&\left(1-2iq^{-\frac{a+2}{2}}\frac{\sin(\frac{a+1}{2}\theta)\sin(\frac{a+3}{2}\theta)}{\sin(\frac{a+2}{2}\theta)}\right)I_{l_2(1)}
\end{bmatrix},
\end{align*}
where $q=\exp(i\theta)$.
There exists a matrix $M(1)$ if $r=6$.
By substituting $a=1$ and $\theta=\pi/3$, we obtain
\[
 M(1)=
\begin{bmatrix}
-\frac{1}{2}I_{l_2(1)}&-\frac{3i}{4}I_{l_2(1)}\\
-iI_{l_2(1)}&-\frac{1}{2}I_{l_2(1)}
\end{bmatrix}.
\]
If $\theta=5\pi/3$, then
\[
 M(1)=
\begin{bmatrix}
-\frac{1}{2}I_{l_2(1)}&\frac{3i}{4}I_{l_2(1)}\\
iI_{l_2(1)}&-\frac{1}{2}I_{l_2(1)}
\end{bmatrix}.
\]
The eigenvalues of $M(1)$ is $\exp(2\pi i/3)$ and $\exp(4\pi i/3)$ and $M(1)^3=I_{l_2(1)}$ in $PGL_{l_2(1)}(\mathbb{C})$.

We obtain the answer for Masbaum's comments ($2n\geq 6, m=5$) from Lemma~\ref{Yuasalemodd} and Lemma~\ref{Yuasalemeven}.

\begin{THM}\label{Yuasaeven}
For any $2n\geq 6$, $m\geq 5$ and $m\neq6$, 
$N_m$ has infinite index in $\mathcal{M}(0,2n)$.
\end{THM}

\begin{proof}
From Proposition~\ref{Masbaumprop}, 
$\rho(\sigma_i^m)$ is trivial if
$q$ is an $m$-th root of unity ($m\colon$ even) or a $2m$-th root of unity ($m\colon$ odd).
By Lemma~\ref{Yuasalemodd} and \ref{Yuasalemeven}, 
there exists a primitive root of unity $q$ satisfying $\rho(\sigma_i^m)=\text{Id}$ and $\rho_T((\sigma_1\sigma_2\dots\sigma_{n-1})^n\sigma_n(\sigma_1\sigma_2\dots\sigma_{n-1})^{-n}\sigma_n^{-1})$ has infinite order. 
For example, $q=\exp(3\pi i/5)$ for $m=5$.
\end{proof}

Consequently, we completely determined the index of a subgroup $N_m(0,2n)$ in $\mathcal{M}(0,2n)$ is infinite or not for any $n$ and $m$.

As in \cite[Theorem~C]{Stylianakis17}, 
this theorem implies some results about $N_{2m}(0,2n+1)$ in $\mathcal{M}(0,2n+1)$. 
$N_{2m}(0,2n)$ is a subgroup of the pure mapping class group $\mathcal{PM}(0,2n)$ of $2n$-punctured sphere. 
A forgetful homomorphism $\mathcal{PM}(0,2n+1)\to\mathcal{PM}(0,2n)$ between pure mapping class groups projects $N_{2m}(0,2n+1)$ on $N_{2m}(0,2n)$. 
Therefore, $N_{2m}(0,2n+1)$ has infinite index in $\mathcal{PM}(0,2n+1)$ if $N_{2m}(0,2n)$ infinite index in $\mathcal{PM}(0,2n)$. 
Then, $N_{2m}(0,2n+1)$ has infinite index in $\mathcal{PM}(0,2n+1)$. 
$N_{2m}(0,2n+1)$ has also infinite index in $\mathcal{M}(0,2n+1)$ because $\mathcal{PM}(0,2n+1)$ is finite index subgroup of $\mathcal{M}(0,2n+1)$.
We obtain the following corollary from Theorem~\ref{Masbaum}, Theorem~\ref{Stylianakis}, and Theorem~\ref{Humphriesinfty}. 
\begin{COR}[{\cite[Theorem~C]{Stylianakis17}}]
For $2n\geq 4$ and $2m\geq 6$, or $2n\geq 6$ and $2m\geq 4$, 
$N_{2m}(2n+1)$ has infinite order in $\mathcal{M}(0,2n+1)$.
\end{COR}
We note that the proof of Theorem~\ref{Masbaum} and Theorem~\ref{Yuasaeven} can be applied directly in the case of $N_{2m}(0,2n)$ in $\mathcal{PM}(0,2n)$ because $\sigma_1^2\sigma_2^{-2}$ and $(\sigma_1\sigma_2\dots\sigma_{n-1})^{n}\sigma_n(\sigma_1\sigma_2\dots\sigma_{n-1})^{-n}\sigma_n^{-1}$ are elements in $\mathcal{PM}(0,2n)$. 

\begin{QUE}
Is the index of $N_{2m+1}(0,2n+1)$ in $\mathcal{M}(0,2n+1)$ either finite or infinite.
\end{QUE}

\section{The normal closure of powers of Dehn twists in hyperelliptic mapping class groups}\label{hyperelliptic}
As an application of our calculations, we consider the normal closure of powers of certain types of Dehn twists in hyperelliptic mapping class groups. 
Let $\Sigma_g$ be an oriented closed surface of genus $g$. 
The mapping class group $\mathcal{M}(g,0)$ of $\Sigma_g$ is the group of the isotopy classes of orientation-preserving homeomorphisms of $\Sigma_g$. 
We fix a hyperelliptic involution $\iota$ on $\Sigma$, that is, an order $2$ homeomorphism of $\Sigma_g$ acting on $H_1(\Sigma_g)$ by $-\operatorname{Id}$. 
The hyperelliptic mapping class group of $\Sigma_g$ is the centralizer of the isotopy class of $\iota$ in $\mathcal{M}(g,0)$. 
We denote it by $\Delta(g,0)$. 
The mapping class of a $(2g+2)$-punctured sphere and the hyperelliptic mapping class group of $\Sigma_g$ are related by the Birman-Hilden theorem:
\begin{THM}[Birman-Hilden~\cite{BirmanHilden73}]\label{BH}
\[
\Delta(g,0)/\langle\iota\rangle \cong \mathcal{M}(0,2g+2).
\]
\end{THM}
Let $T_{c}$ be a right-handed Dehn twist along a simple closed curve (SCC) $c$ on $\Sigma_g$.
We consider Dehn twists along {\em symmetric non-separating} and {\em symmetric separating simple closed curves}. 
A simple closed curve is {\em symmetric} if $\iota$ preserves the isotopy class of $c$.
We denote a symmetric simple closed curve by an $\iota$-SCC. 
We note that $T_c\in\Delta(g,0)$ if $c$ is an $\iota$-SCC.

Let us study about the normal closure of powers of such Dehn twists using the Birman-Hilden theorem and the projective representation of $\mathcal{M}(0,2g+2)$. 
A projective representation $\Delta(g,0)\to\mathcal{S}_q(1^{\otimes 2g+2})$ is obtained by the composition $\Delta(g,0)\to\mathcal{M}(0,2g+2)$ and $\rho\colon\mathcal{M}(0,2g+2)\to\mathcal{S}_q(1^{\otimes 2g+2})$. 
We denote it also by $rho$.
Let $\mathcal{N}_m^{\iota}(g,0)$ be the normal closure of the $m$-th power of a Dehn twist along a non-separating $\iota$-SCC in $\Delta(g,0)$, and $\hat{\mathcal{N}}_m^{\iota}(g,0)$ the normal subgroup generated by all $\iota$-SCCs in $\Delta(g,0)$. 
We denote them $\mathcal{N}_m^{\iota}$ and $\hat{\mathcal{N}}_m^{\iota}$ for simplicity.
The isomorphism of Theorem~\ref{BH} sends $T_c\in\Delta(g,0)$ to a half twist permuting two punctures in $\mathcal{M}(0,2g+2)$ if $c$ is a non-separating $\iota$-SCC.
Therefore, Theorem~\ref{Stylianakis} and Theorem~\ref{Humphriesinfty} implies Corollary~\ref{Stylianakiscor}.
We consider a generalization of Corollary~\ref{Stylianakiscor} by replacing $\mathcal{N}_m^{\iota}$ by $\hat{\mathcal{N}}_m^{\iota}$. 
Let $\delta_h$ be a separating $\iota$-SCC which decomposes $\Sigma_g$ into two subsurface of genus $h$ and $g-h$ ($h\in\{\,1,2,\dots,g-1\,\}$).
The Dehn twist $T_{\delta_h}\in\Delta_g$ corresponds to $(\sigma_1\sigma_2\dots\sigma_{2h})^{4h+2}$ in $\mathcal{M}(0,2g+2)$.
We already calculate $\rho$ of such elements in (\ref{fulltwist}).
\begin{align*}
\rho((\sigma_1\sigma_2\dots\sigma_{2h})^{(4h+2)m})(\beta_T(a_1,a_2,\dots,a_{2g-1}))
&=q^{\frac{3m(2h+1)}{2}}q^{-\frac{m(a_{2h}^2+2a_{2h})}{2}}\beta_T(a_1,a_2,\dots,a_{2g-1})\\
&=q^{3mh}q^{-\frac{m(a_{2h}-1)(a_{2h}+3)}{2}}\beta_T(a_1,a_2,\dots,a_{2g-1})
\end{align*}
for $0< h< g$.
We only have to consider in the case of $0< h<\lfloor g/2\rfloor$ to determine when $(\sigma_1\sigma_2\dots\sigma_{2h})^{(4h+2)m}$ send to identity in $PGL(\mathbb{C})$ by $\rho$ at a root of unity.
\begin{PROP}\label{sepprop}
Let $q$ be a primitive $r$-th root of unity and $\rho\colon\Delta(g,0)\to\mathcal{S}_q(1^{\otimes 2g+2})$ the projective representation.
Then, 
\begin{enumerate}
\item $\rho(T_{\delta_h}^m)=\operatorname{Id}$ for any $g\geq 2$ and $1\leq h\leq \lfloor g/2 \rfloor$ when $r=4$,
\item $\rho(T_{\delta_h}^m)=\operatorname{Id}$ if $q^{6m}=1$ for any $g\geq 2$ and $1\leq h\leq \lfloor g/2 \rfloor$ when $r=5,6$,
\item $\rho(T_{\delta_1}^m)=\operatorname{Id}$ if $q^{6m}=1$ for $g=2,3$ when $r\geq 7$.
\item $\rho(T_{\delta_h}^m)=\operatorname{Id}$ if $q^{2m}=1$ for any $g\geq 4$ and $1\leq h\leq \lfloor g/2 \rfloor$ when $r\geq 7$.
\end{enumerate}
\end{PROP}
\begin{proof}
Let $q$ be a primitive $r$-th root of unity such that $r\geq 4$ and $g\geq 2$. 
It is only necessary to consider the value of $\rho(T_{\delta_h}^m)$ for $h=1,2,\dots,\lfloor g/2 \rfloor$. 
The $q$-admissibility implies that the possible values of $a_{2h}$ is in $\{1,3,\dots ,\min\{r-2, 2g-1\}\}$ if $r$ is odd, $a_{2h}\in\{1,3,\dots,\min\{r-3,2g-1\}\}$ if $r$ is even. 
If $r=4$, then $\rho(T_{\delta_h}^m)$ is always the identity in $PGL(\mathbb{C})$ for any $h$. 
If $r= 5,6$, then the value of $a_{2h}$ is $1$ or $3$ for any $h$.  
Therefore, $\rho(T_{\delta_h}^m)=\operatorname{Id}$ if and only if $q^{3mh}=q^{3mh}q^{-6m}$, that is, $q^{6m}=1$.
We assume $r\geq 7$, $g\geq 4$, and $r-2\leq 2g-1$. 
Let us fix $h_0\in\{1,2,\dots,\lfloor g/2 \rfloor\}$. 
Then, $\rho(T_{\delta_{h'}}^m)=\operatorname{Id}$ if and only if 
$q^{3mh'}q^{-\frac{m(a-1)(a+3)}{2}}=q^{3mh'}q^{-\frac{m(a+1)(a+5)}{2}}$ for $a=1,3,\dots r-2$ when $r$ is odd, or for $a=1,3,\dots r-3$ when $r$ is even. Hence $q^{2m}=1$ and this condition is independent of a choice of $h'$. 
The case of $r\geq 7$, $g\geq 4$, and $2g-1\leq r-2$ is the same. 
\end{proof}
Let $c$ be a non-separating $\iota$-SCC and $\delta_h$ a separating $\iota$-SCC bounding genus $h$ subsurface on $\Sigma_g$. 
For non-negative integers $k$ and $l$, we denote the normal closure of $\{\,T_c^{k},T_{\delta_h}^l\mid h=1,2,\dots,\lfloor g/2 \rfloor\,\}$ in $\Delta(g,0)$ by $\mathcal{N}^{\iota}_{(k,l)}$. 
We note that $\mathcal{N}_m^{\iota}=\mathcal{N}_{(m,0)}^{\iota}$ and $\hat{\mathcal{N}}_m^{\iota}=\mathcal{N}_{(m,m)}^{\iota}$. 
If there exists positive integers $a$ and $b$ such that $k=ak'$ and $l=bl'$, then $\mathcal{N}_{(k,l)}^{\iota}\subset\mathcal{N}_{(k',l')}^{\iota}$. 
\begin{THM}\label{klsubgroup}\ 
\begin{enumerate}
\item For $g= 2,3$, $\mathcal{N}^{\iota}_{(6m+3,2m+1)}$ and $\mathcal{N}^{\iota}_{(6m+9,2m+3)}$ has infinite index in $\Delta(g,0)$ if $m\geq 1$.
\item For $g= 2,3$, $\mathcal{N}^{\iota}_{(6m,m)}$ and $\mathcal{N}^{\iota}_{(6m+6,m+1)}$ has infinite index in $\Delta(g,0)$ if $m\geq 2$.
\item For $g\geq 4$, $\mathcal{N}^{\iota}_{(2m,m)}$ and $\mathcal{N}^{\iota}_{(2m+1,2m+1)}=\hat{\mathcal{N}}_{2m+1}$ has infinite index in $\Delta(g,0)$.
\end{enumerate}
\end{THM} 
\begin{proof}
We consider a situation that $\rho(T_c^{k})$ is the identity in $PGL(\mathcal{S}_q(1^{\otimes 2g+2}))$, that is, $q^{k}=(-1)^k$ by Proposition~{\ref{Masbaum}}. 
If $g=2,3$ and $k=6m+3$,
we take the primitive $2(6m+3)$-th root of unity in Lemma~\ref{Yuasalemodd} or Lemma~\ref{Yuasalemeven}. 
Then, $\rho(T_{\delta_h}^{(2m+1)})=\operatorname{Id}$ because of Proposition~\ref{sepprop}~(3) and $q^{6(2m+1)}=1$. 
This proves $\mathcal{N}^{\iota}_{(6m+3,2m+1)}$ has infinite index in $\Delta(g,0)$. 
In the case of $\mathcal{N}^{\iota}_{(6m+9,2m+3)}$, we can prove in the exact same way.
If $g=2,3$ and $k=6m$, 
we take the primitive $6m$-th root of unity in Lemma~\ref{Yuasalemodd} or Lemma~\ref{Yuasalemeven}. 
Then, $\rho(T_{\delta_h}^m)=\operatorname{Id}$ because of Proposition~\ref{sepprop}~(3) and $q^{6m}=1$. 
For $k=6m+6$, we can prove in the same way.
If $g\geq 4$ and $k=2m, 2m+1$, we can prove in the similar way by using Proposition~\ref{sepprop}~(4).
\end{proof}

\begin{COR}\
For any $g\geq 2$, $\hat{\mathcal{N}}_m^{\iota}$ has infinite index in $\Delta(g,0)$ if $m\geq 5$ and $m\neq 6$. 
\end{COR}
\begin{proof}
We remark that some cases are derived from Theorem~\ref{klsubgroup}. 
It is easily shown by the same way to the proof of Theorem~\ref{klsubgroup}.
If $m\geq 8$ is an even integer (resp. $m\geq 5$ is an odd integer), we take the primitive $m$-th (resp. $2m$-th) root of unity in Lemma~{\ref{Yuasalemodd}} or Lemma~{\ref{Yuasalemodd}}, and use Proposition~\ref{sepprop}~(3) and (4). 
\end{proof}

\bibliographystyle{amsalpha}
\bibliography{puncturedMCG_power}
\end{document}